\documentclass[11pt,reqno]{amsart}

\usepackage{amsfonts,color}
\usepackage{graphics,graphicx,array}
\usepackage{indentfirst}
\usepackage{cite}
\usepackage{latexsym,hyperref}
\usepackage{amsmath}
\usepackage{amssymb}
\usepackage[dvips]{epsfig}
\usepackage{amscd}

\allowdisplaybreaks
\newtheorem{theorem}{Theorem}[section]
\newtheorem{lemma}{Lemma}[section]

\newtheorem{proposition}{Proposition}[section]

\theoremstyle{definition}
\newtheorem{definition}{Definition}[section]

\theoremstyle{remark}
\newtheorem{remark}{Remark}[section]

\renewcommand{\div}{ {\rm div}}
\newcommand{\na}{\nabla }

\newcommand{\de}{\delta}
\newcommand{\eps}{\epsilon}

\newcommand{\la}{\label}
\newcommand{\si}{\sigma}
\newcommand{\Ga}{\Gamma}

\allowdisplaybreaks

\newcommand{\bnn}{\begin{eqnarray*}}
\newcommand{\enn}{\end{eqnarray*}}

\newcommand{\ba}{\begin{aligned}}
\newcommand{\ea}{\end{aligned}}
\newcommand{\be}{\begin{equation}}
\newcommand{\ee}{\end{equation}}

\def\b{\beta}
\def\p{\partial}
\def\th{\theta}
\def\lap{\triangle}

\def\om{\Omega}
\def\r{\mathbb{R}^{3}}

\def\dist{\text{dist}}

\makeatletter
\@addtoreset{equation}{section}
\makeatother

\title[Vanishing viscosity near Onsager's critical regularity]{A Kato-type criterion for vanishing viscosity near the Onsager's critical regularity}

\author[R. M. Chen]{Robin Ming Chen}
\address{Department of Mathematics, University of Pittsburgh, Pittsburgh, PA 15260.} 
\email{mingchen@pitt.edu}
\author[Z. L. Liang]{Zhilei Liang}
\address{School of Economic Mathematics, Southwestern  University of Finance and Economics, Chengdu  611130,  China.} 
\email{zhilei0592@gmail.com}
\author[D. Wang]{Dehua Wang}
\address{Department of Mathematics, University of Pittsburgh, Pittsburgh, PA 15260.} 
\email{dwang@math.pitt.edu}

\begin{document}

\maketitle

 \begin{abstract}   
 We consider a vanishing viscosity sequence of weak solutions of the three-dimensional Navier--Stokes equations on a bounded domain. In a seminal paper \cite{Kato1984} Kato showed that for sufficiently regular solutions, the vanishing viscosity limit is equivalent to having vanishing viscous dissipation in a boundary layer of width proportional to the viscosity. We prove that Kato's criterion holds for H\"older continuous solutions with the regularity index arbitrarily close to the Onsager's critical exponent through a new boundary layer  foliation and a global mollification. 
 \end{abstract}

\setcounter{tocdepth}{1} 

\section{Introduction}

The motion of an incompressible viscous  fluid  with   constant density    is governed by the following Navier-Stokes equations:   
\begin{equation}\label{1ns}
\left\{\ba
& \p_{t}u^{\nu} +{\rm div}(u^{\nu}\otimes u^{\nu} )+\na P^{\nu}=\nu \lap u^{\nu},\\
& {\rm div} u^{\nu}=0,\\
&u^{\nu}(x,0)=u_{0}^{\nu}(x),\ea \right.
\end{equation}
 where the constant  $\nu>0$ denotes the viscosity  of the fluid,   the unknown functions   $ u^{\nu}$  and $ P^{\nu}$  are   the    velocity field  and  pressure, respectively. Here the superscript $\nu$ is used on all the unknowns to emphasize the dependence on the viscosity. The  Navier-Stokes equations at zero viscosity $\nu = 0$ formally become the Euler equations:
 \begin{equation}\label{1e}
\left\{\ba
& \p_{t}u  +{\rm div}(u \otimes u)+\na P =0,  \\
& {\rm div} u=0, \\
&u(x,0)=u_{0}(x).\\
\ea \right.
\end{equation}
 
An important problem in the study of incompressible hydrodynamics is the vanishing viscosity
limit from the Navier-Stokes equations \eqref{1ns} to the Euler equations \eqref{1e}, which is naturally associated with the physical phenomena of turbulence and of boundary layers. On domains without boundary, such a problem is well-understood: given a strong Euler solution $u^E \in C^1$, the Leray--Hopf solutions $u^\nu$ of \eqref{1ns} converge strongly in the energy space $L^\infty_tL^2_x$ to $u$ as $\nu\to0$; see, for exmple, \cite{BardosTiti2013}.

In the presence of boundary, on the other hand, systems \eqref{1ns} and \eqref{1e} considered in a bounded domain $\Omega$ are supplemented with the no-slip boundary condition $u|_{\partial\Omega} = 0$ and slip  boundary condition $(u \cdot n) |_{\partial\Omega} = 0$, respectively, with $n$ the unit outward normal of the boundary $\partial\Omega$.
The mismatch of the boundary conditions leads to the phenomenon of boundary layer separation. Establishing the vanishing viscosity limit 
in the energy space $L^\infty_tL^2_x$
in this case is much less understood.  A well-known result of Kato \cite{Kato1984} states that,
for a strong Euler solution $u^E$, the vanishing viscosity limit
\[
u^{\nu} \to u^E \quad \text{in }\ L^\infty(0, T; L^2(\Omega))
\]
holds if and only if 
\be\label{Kato}
\nu \int^T_0 \|\nabla u^\nu\|^2_{L^2(\Gamma_{c\nu})} \,dt \to 0 \quad \text{as }\ \nu \to 0, 
\ee
where $\Gamma_{c\nu}$ is a very thin boundary layer of width proportional to $\nu$.

Kato's theory is by nature conditional. Many of the known results on strong inviscid limits are also conditioned on special properties of the solutions \cite{bardos2007euler,BardosTiti2013,constantin2015inviscid,constantin2017remarks,Kelliher2007,kelliher2008vanishing,temam1997behavior,wang2001kato}. Some unconditional strong convergence results do exist, but with additional assumptions on the data like real analyticity \cite{caflisch1998zero}, vanishing of the initial vorticity near the boundary \cite{maekawa2014inviscid}, or special symmetry of the flow \cite{lopes2008vanishing,lopes2008physicad,mazzucato2008vanishing,kelliher2009vanishing}. These results are for short time and for laminar flows close to a smooth Euler solution when there is no boundary layer separation or other characteristic turbulent behavior.

The vanishing viscosity for turbulent flows faces serious challenges and remains widely open. 
 Little is known about the inviscid limit even when  a strong Euler solution exists for a short time. Therefore it is natural to consider the weak Euler solutions for the vanishing viscosity limit.
One type of such weak Euler solutions is the measure valued solutions \cite{diperna1987oscillations}. In some recent works \cite{constantin2018remarks,drivas2018remarks} the authors describe sufficient conditions in terms of interior structure functions under which the weak $L^\infty_tL^2_x$ solutions $u^E$ of the Euler equations can be obtained as weak $L^2_{t,x}$ limits of $u^\nu$. In \cite{constantin2019vorticity}, the authors further extend the result of \cite{constantin2018remarks} to allow certain interior vorticity concentration.

 The classical result of Kato \cite{Kato1984} indicates that anomalous energy dissipation leads to the failure of the inviscid limit to a strong ($C^1$) Euler solution; while  the issue that weak Euler solutions may arise from the inviscid limit is  closely related to the Onsager's conjecture (see, for example \cite{BuckmasterVicol2019,drivas2018remarks} and the references therein).
It has been made a precise statement that the critical Onsager's H\"older regularity exponent is ${1/3}$, below which the Euler equations  become non-conservative \cite{De_Lellis_2012,De_Lellis_2014,Isett_2018,Isett_2017,Buckmaster_2019}, while above ${1/3}$  the energy conservation can be justified \cite{Eyink_1994,Constantin_1994,Duchon_1999,Cheskidov_2008}. In the works of \cite{bardos2018onsager1,DrivasNguyen2018} the authors derive sufficient conditions for $C^\alpha$ solutions under which the global viscous dissipation vanishes in the inviscid limit for Leray--Hopf solutions $u^\nu$, with an emphasis on the behavior or solutions near the boundary. In particular in \cite{DrivasNguyen2018} a Kato-type criterion on the vanishing of the energy dissipation rate in a thin boundary layer of thickness $O(\nu^\beta)$ is proposed, among other regularity conditions, where $\beta = {3\over4}+\epsilon$ near the critical Onsager threshold $\alpha = {1\over3}+\epsilon$. Note that the boundary layer in the result of \cite{DrivasNguyen2018} is thicker than that of Kato's. 

The main goal of this paper is to bridge the gap between the original result of Kato \cite{Kato1984} for strong $C^1$ Euler solutions and the result of \cite{DrivasNguyen2018} for weak $C^{\alpha}$ Euler solutions. Specifically, we will show that under certain $\nu$-dependent assumptions on the family of solutions of \eqref{1ns}, a Kato-type result with boundary layer of thickness $O(\nu)$ holds for weak Euler solutions up to Onsager-critical spatial regularity $\alpha = {1\over3}+$. See Table \ref{table1} below. 

\begin{table}[h!]
\centering
\begin{tabular}{| m{8.8em} | m{9em} |c|c|}
\hline\rule[-4mm]{0cm}{10mm}
\ & \hspace{.3cm} NS solution $u^\nu$ \hspace{.3cm} & \hspace{.1cm} Euler solution $u^E$ \hspace{.1cm} & \hspace{.1cm} boundary layer \hspace{.1cm} \\ 
 \hline\rule[-4mm]{0cm}{10mm} Kato \cite{Kato1984} & \hspace{.4cm} Leray--Hopf  & $C^1(\overline{\om})$ & $O(\nu)$ \\
\hline\rule[-4mm]{0cm}{10mm} Drivas--Nguyen \cite{DrivasNguyen2018} & $C^{1/3 +}_{\text{loc}}(\om)$ with boundary regularity & $C^{1/3 +}_{\text{loc}}(\om)$ & $O(\nu^{3/4 +})$ \\ \hline\rule[-4mm]{0cm}{10mm} Our result & $C^{1/3 +}_{\text{loc}}(\om)$ with boundary regularity & $C^{1/3 +}_{\text{loc}}(\om)$ & $O(\nu)$ \\ \hline
\end{tabular}\bigskip
\caption{Comparison of results} \vspace{-.5cm}
\label{table1}
\end{table}

Let us recall the classical existence results of Leray \cite{Leray34} and Hopf \cite{Hopf51}. For a divergence-free function $u_{0}\in L^{2},$   problem  \eqref{1ns} has a weak solution  $u\in C\left(0,T;L^{2}\right) \cap L^{2}\left(0,T;H^{1}(\om)\right)$  in a bounded smooth domain $\om$ for any $T<\infty$.   Additionally, $u$ is divergence-free and  the following energy inequality holds
 \be\la{0} \frac{1}{2}\int_{\om}|u^{\nu}|^{2}dx+\nu\int_{0}^{t}\int_{\om}|\na u^{\nu}|^{2}dxds\le \frac{1}{2}\int_{\om}|u^{\nu}_{0}|^{2} dx,\quad a.e.\,\,\, t\in (0, T).\ee   Such a weak solution is called {\it Leray--Hopf weak solution}.

%

Next we introduce some notation. For some (small) $h>0,$ we define   
\be\la{d}  \om^{h} :=\{x\in \om,\,\,\,\dist(x,\,\p\om)>h\}\quad{\rm and}\quad \Ga_{h} :=\om\backslash \om^{h}.\ee
Also, we introduce  the   Besov space $B_{p}^{\alpha,\infty}(\om)$ which consists of  measurable  functions  with the norm
\be\la{b} \|f\|_{B_{p}^{\alpha,\infty}(\om)}:=\|f\|_{L^{p}(\om)}+\sup_{y\in \r}\frac{\|f(\cdot+y)-f(\cdot)\|_{L^{p}(\om\cap (\om-\{y\}))}}{|y|^{\alpha}},\quad p\ge 1, \ \alpha \in (0,1).\ee

We further denote $H(\om)$ to be the completion in $L^2(\om)$ of the space $\{ v \in C^\infty_c(\om; \mathbb R^3):\ \div v = 0 \}$, and recall the following definition of the weak Euler solutions (see, for example \cite{DrivasNguyen2018}).
\begin{definition}\label{def euler}
Let  $\om\subset \r$ be a bounded domain with $C^2$ boundary. We say the pair $(u, P)$ is a weak Euler solution to \eqref{1e} on $\om \times (0, T)$ if $u\in C_w(0, T; H(\om))$, $P \in L^1_{\text{loc}}(\om \times (0, T))$ and for all test vector fields $\varphi \in C^\infty_0(\om\times(0, T))$ it holds that
\begin{equation*}
\int^T_0\int_\om \left( u\cdot \partial_t \varphi + u \otimes u : \nabla \varphi + P \nabla \cdot \varphi \right)\,dxdt = 0.
\end{equation*}
\end{definition}

Our main result is stated in the following theorem.
\begin{theorem}\la{thm0}  
Let  $\om\subset \r$ be a bounded domain with $C^2$ boundary.  Let $\{ u^\nu \}_{\nu>0}$ be a sequence of Leray--Hopf weak solutions to \eqref{1ns} with initial data $u_{0}^{\nu}$ and suppose that $u_{0}^{\nu} \to u_0$ in $L^{2}(\om)$ as $\nu \to 0$. Assume in addition that 
  \be\la{h1}    \ba &  u^{\nu}\,\,{is\,\, uniformly\,\, in} \,\,\nu \,\, {bounded\,\, in}\,\,  L^{3}\left(0,T;B_{3}^{\alpha,\infty}(\om^{\nu})\right)\,\,\,{for\,\,some}\,\,\,\alpha\in  \left(\frac{1}{3},\,1 \right), \ea\ee 
 \be\ba\la{02h2} \left\{\ba&   u^{\nu} \,\,{is\,\, uniformly \,\,in}\,\,  \nu \,\,{bounded\,\, in}\, \,  L^{4}\left(0,T;L^{\infty}\left(\Ga_{4\nu}\right)\right),\\
&  P^{\nu}\,\,{\ is \,\, uniformly \,\, in}\,\,  \nu \,\,{bounded\,\, in}\, \,  L^{2}\left(0,T;L^{\infty}\left(\Ga_{4\nu}\right)\right).\ea\right.\ea\ee
Then,  if 
\be\la{0j1}  \lim_{\nu\rightarrow0}\nu \int_{0}^{T}\int_{\Ga_{4\nu}}|\na u^{\nu}|^{2}dxdt=0,\ee
we have that the global   viscous dissipation vanishes, i.e.,
\begin{equation}\la{00}
\lim_{\nu\rightarrow0}\nu\int_{0}^{T}\int_{\om}|\na u^{\nu}|^{2}dxdt=0,
\end{equation}
and moreover,  $u^{\nu}$ converges locally  in $L^{3}(0,T;L^{3}(\om))$,  up to some  subsequence,  to a weak solution of the Euler equations  \eqref{1e}. 
\end{theorem}
\begin{remark}\label{rk_Kato}
Condition \eqref{0j1} recovers Kato's criterion \eqref{Kato}, but now in the framework of weak solutions with $C^\alpha$ regularity, where $\alpha$ can be taken arbitrarily close to the Onsager's critical exponent, cf. \eqref{h1}. 
\end{remark}
\begin{remark}\label{rk_DN}
As  pointed out in \cite{DrivasNguyen2018}, violation of conditions \eqref{h1} -- \eqref{0j1} is responsible for global dissipation to persist in the vanishing viscosity limit. More precisely, violation of \eqref{h1} corresponds to a failure of uniform interior regularity, which is required for anomalous dissipation in domains without boundaries; see, for example \cite{Constantin_1994}. On the other hand, conditions \eqref{02h2} -- \eqref{0j1} provide new mechanisms for anomalous dissipation in wall-bounded flows.
\end{remark}

The basic idea in \cite{DrivasNguyen2018} is separation and regularization: applying a cut-off function to separate the boundary from the interior domain, and mollifying the interior velocity field. This introduces two length scales: the thickness $h$ of the  boundary layer and the mollification scale $\epsilon$. With this localization, the resolved dissipation is bounded by $\nu \epsilon^{2(\alpha - 1)} \int^t_0 \|u^\nu\|^2_{B^{\alpha,\infty}_3(\Omega^h)}$; see \cite[Section 2.1]{DrivasNguyen2018}. Recalling the natural constraint that $\epsilon \le h$, imposing appropriate interior regularity assumption on the solution, and setting $h \sim \nu^\beta$, the above estimate translates to $\nu^{1 + 2\beta(\alpha - 1)}$. In order for this to vanish at the inviscid limit $\nu \to 0$ one has to require that 
\be\label{DN beta}
1 + 2\beta(\alpha - 1) > 0, 
\ee
which, at the Onsager's critical regularity $\alpha = {1\over3}+$, returns $\beta = {3\over4}+$.
 
The obvious obstacle in the above approach to get to the $O(\nu)$ boundary layer thickness lies in the strong constraint between the two scales $\epsilon$ and $h$. In other words, if one can find a way to ``free up" the choice for $\epsilon$ so as to improve the mollification, then it is reasonable to hope to obtain a thinner boundary layer. 

Motivated by a recent work of the authors \cite{Chen2019}, where a global mollification was introduced, we design a new localization technique with additional treatment near the boundary. In particular, we will start with a boundary layer of the type as in \cite{DrivasNguyen2018} and perform a further foliation within that boundary layer, mollify the solution with different scales in each leaf of the foliation, and then glue everything together by a partition of unity. Such a new type of mollification generates additional cancelation effects in estimating the resolved dissipation, allowing one to reach the $O(\nu)$ boundary layer. Moreover, using the same idea, we also show that as the solution becomes more regular (corresponding to increasing $\alpha$), the regularity requirement \eqref{02h2}  near the boundary can be relaxed, and the the boundary layer in \eqref{0j1} is allowed to be thinner, cf. Theorem \ref{thm2}.

The rest of the paper is organized as follows. In Section \ref{sec prep} we briefly recall some needed analytical results, and introduce the boundary layer foliation. In Section \ref{sec reg} we define the mollification and use that to regularize the system. By testing the resolved system with suitable test function we prove the balance of the resolved energy, from which we proceed in Section \ref{sec proof} to give the proof of Theorem \ref{thm0}. Finally in Section \ref{sec higher reg} we extend the result of Theorem \ref{thm0} to the case when solutions are more regular.

\section{Preliminaries}\label{sec prep}
\subsection{Commutator estimates}
Recall the standard  mollification for  the  function $f$ 
 \be\la{r2}\overline{f_{\eps}}(x) :=\int_{B_{\eps}(0)} f(x-y)\eta_{\eps}(y)dy, \quad \forall\,\,\,x\in \om^{\eps},\ee
 with $\eta_{\eps}$ being the standard mollifier of width $\eps.$
Straightforward calculation gives     \bnn\ba \na \overline{f_{\eps}}(x)
&=-\int_{B_{\eps}(0)} \eta_{\eps}(y)\na_{y}( f (x -y)-f(x) )dy={1 \over \eps} \int_{B_1(0)}\na \eta(y) ( f (x - \eps y)-f(x)) dy,\ea\enn
and hence, for $f \in B^{\alpha,\infty}_r$ with $r\in [1,\infty]$,   
\be\la{2.12} \|\na \overline{f_{\eps}}\|_{L^{r}(\om^{\eps})}  \le \eps^{\alpha-1}\|f\|_{B_{r}^{\alpha,\infty}(\om)}.
\ee
Similarly,
\be\la{b5}\ba \|\overline{f_{\eps}} -f\|_{L^{r}(\om^{\eps})}
 \le C \eps^{\alpha}\|f\|_{B_{r}^{\alpha,\infty}(\om)}.\ea\ee
\begin{lemma} \la{lem2.1}   Let  $f\in B_{r_{1}}^{\alpha,\infty}(\om),\,g\in B_{r_{2}}^{\alpha,\infty}(\om)$,  and let  $1\le r,\,r_{1},\,r_{2}< \infty,$  $ {1\over r_{1}}+ {1\over r_{2}}= {1\over r}.$ Then there exists some $C>0$ such that the following hods
\be\la{q2} \|\overline{(f\otimes g)_{\eps}}-\overline{f_{\eps}}  \otimes \overline{g_{\eps}}\|_{L^{r}( \om^{\eps})}\le C \eps^{2 \alpha}\|f\|_{B_{r_{1}}^{\alpha,\infty}(\om)}\|g\|_{B_{r_{2}}^{\alpha,\infty}(\om)}.\ee  \end{lemma}
\begin{proof}   
Inequality \eqref{q2} is nothing but the commutator estimate in \cite{Constantin_1994}. Here we give an outline of the proof.   

Let    $\de f (x,y) := f(x -y)-f(x).$   By \eqref{r2} we compute   for every $x\in \om^{\eps}$, 
   \be\ba\la{q3.16}   & \overline{(f\otimes g)_{\eps}}-\overline{f_{\eps}}  \otimes \overline{g_{\eps}} \\
  & =\int_{B_{\eps}(x)} \de f(x,y)\otimes \de g(x,y) \eta_{\eps}(y)dy   -\left(\int_{B_{\eps}(x)} \de f(x,y)\eta_{\eps}(y)dy\right)\otimes \left(\int_{B_{\eps}(x)} \de g(x,y)\eta_{\eps}(y)dy\right)\\
  &\le \left(\int_{B_{\eps}(x)} |\de f(x,y)|^{r_{1}}\eta_{\eps}(y)dy\right)^{\frac{1}{r_{1}}}\left( \int_{B_{\eps}(x)}|\de g(x,y)|^{r_{2}}\eta_{\eps}(y)dy\right)^{\frac{1}{r_{2}}}+\left|\overline{f_{\eps}}-f \right| |\overline{g_{\eps}}-g|.\ea\ee  Integrating  \eqref{q3.16}  over $\om^{\eps}$ gives 
\bnn\ba  &\int_{\om^{\eps}} \left|\overline{(f\otimes g)_{\eps}}-\overline{f_{\eps}}  \otimes \overline{g_{\eps}}\right|^{r}dx \\
  &\le C \left(\int_{\mathbb{R}^{3}}\eta_{\eps}(y)\int_{\om^{\eps}} |\de f(x,y)|^{r_{1}}dxdy\right)^{\frac{r}{r_{1}}}\left(\int_{\mathbb{R}^{3}}\eta_{\eps}(y)\int_{\om^{\eps}} |\de g(x,y)|^{r_{1}}dxdy\right)^{\frac{r}{r_{2}}}\\
  &\quad+ C \left\|\overline{f_{\eps}}-f \right\|_{L^{r_{1}}}^{r} \|\overline{g_{\eps}}-g\|_{L^{r_{2}}}^{r}.\ea\enn
 This together with   \eqref{b} and   \eqref{b5} conclude the   desired   \eqref{q2}.
\end{proof}

\subsection{Pressure estimates}
The pressure $P^{\nu}$ appeared in \eqref{1ns} can be deduced from  the velocity $u^{\nu}$ via the Poisson equation 
\[
-\lap  P^{\nu}=\div\div (u^{\nu}\otimes u^{\nu}). 
\]
From \cite[Lemma 2]{Es17},  we have the following
\begin{lemma} \la{lemma2.2} 
Let $p\in (1,\infty).$ Assume that $u^{\nu} \in L^{2p}(\om)$ and $P^{\nu}|_{\p\om} \in L^p(\p\om)$. Then the pressure $P^{\nu} \in L^p(\om)$. In  addition, the following estimate holds,
\be\la{pp}  \|P^{\nu}\|_{L^{p}(\om)} \le C\left( \| P^{\nu}|_{\p\om} \|_{L^p(\p\om)} + \|u^{\nu}\|_{L^{2p}(\om)}^{2} \right).\ee
\end{lemma}
\begin{lemma}[Hardy-type embedding \cite{kuf}]\la{lemma2.3} Let $p\in[1,\infty)$ and $f\in W_{0}^{1,p}(\om).$  Then 
\bnn \left\|\frac{f(x)}{{\rm{dist}} (x,\p\om)} \right\|_{L^{p}(\om)}\le C\|\na f\|_{L^{p}(\om)},\enn  where $C$ depends on $p$ and $\om.$ 
\end{lemma}

\subsection{Boundary layer foliation}\label{subsec BL}
For $\alpha\in (\frac{1}{3},\,\frac{5}{6}),$ we define the following sequence
 \be\la{r0}\ba \b_{0}^{*}=0\quad {\rm and}\quad \b_{n}^{*}=\frac{1}{2(1-\alpha)}\left(1+\frac{1}{3}\b_{n-1}^{*}\right),  \quad n=1,2,3,\cdot\cdot\cdot.\ea\ee  
Clearly,  $\{\b_n^*\}$ is bounded and strictly increasing, and 
\bnn \b_{\infty}^{*} :=\lim_{n\rightarrow \infty}\b_{n}^{*}=\frac{3}{5-6\alpha}>1\quad{\rm if}\quad {1\over 3} < \alpha < \frac{5}{6}.\enn
Hence there exists some finite number
 \be\la{k11} N :=\left\{\ba &N(\alpha), &\alpha\in \left(\frac{1}{3},\frac{1}{2} \right];\\
 &1,&\alpha\in \left(\frac{1}{2},{5\over 6} \right)\ea\right. \ee
 such that \be\la{k10} 0=\b_{0}^{*}<\b_{1}^{*}<\b_{2}^{*}<\cdot\cdot\cdot< \b_{N-1}^{*}\le 1<\b_{N}^{*},\ee

In light of \eqref{r0}--\eqref{k10},  we   define, for any $\alpha\in (\frac{1}{3},1),$ an increasing  sequence  $\{\b_{n}\}_{n=1}^{N}$ such that 
\begin{align}\la{r1a}   
& 0= \b_{0}<\b_{1} <\b_{2}<\cdot\cdot\cdot< \b_{N-1}<\b_{N} :=1\quad{\rm and} \\
& \b_{n}< \frac{1}{2(1-\alpha)}\left(1+\frac{1}{3}\b_{n-1}\right), \label{beta inequ}
\end{align}
where    $N$ is given in  \eqref{k11}.

The purpose of introducing the sequence $\{ \beta_n \}$ is to design the following decomposition of the boundary layer. Note that, when  $n=1$, \eqref{beta inequ} reads $\beta_1 < {1\over 2(1 - \alpha)}$, which agrees with \eqref{DN beta}. Next we decompose the inner region of $\om$ as  
\be\la{rra} 
V_{1} :=  \om^{2\nu^{\b_{1}}},\quad    V_{n} :=\om^{2\nu^{\b_{n}}}-\om^{2\nu^{\b_{n-1}}+2\nu^{\b_{n}}} \,\,{\rm when}\,\, 2\le n\le N.
\ee
See Figure \ref{fig:decomp}.
\begin{figure}[h!]
  \centering
  \includegraphics[page=1,scale=0.7]{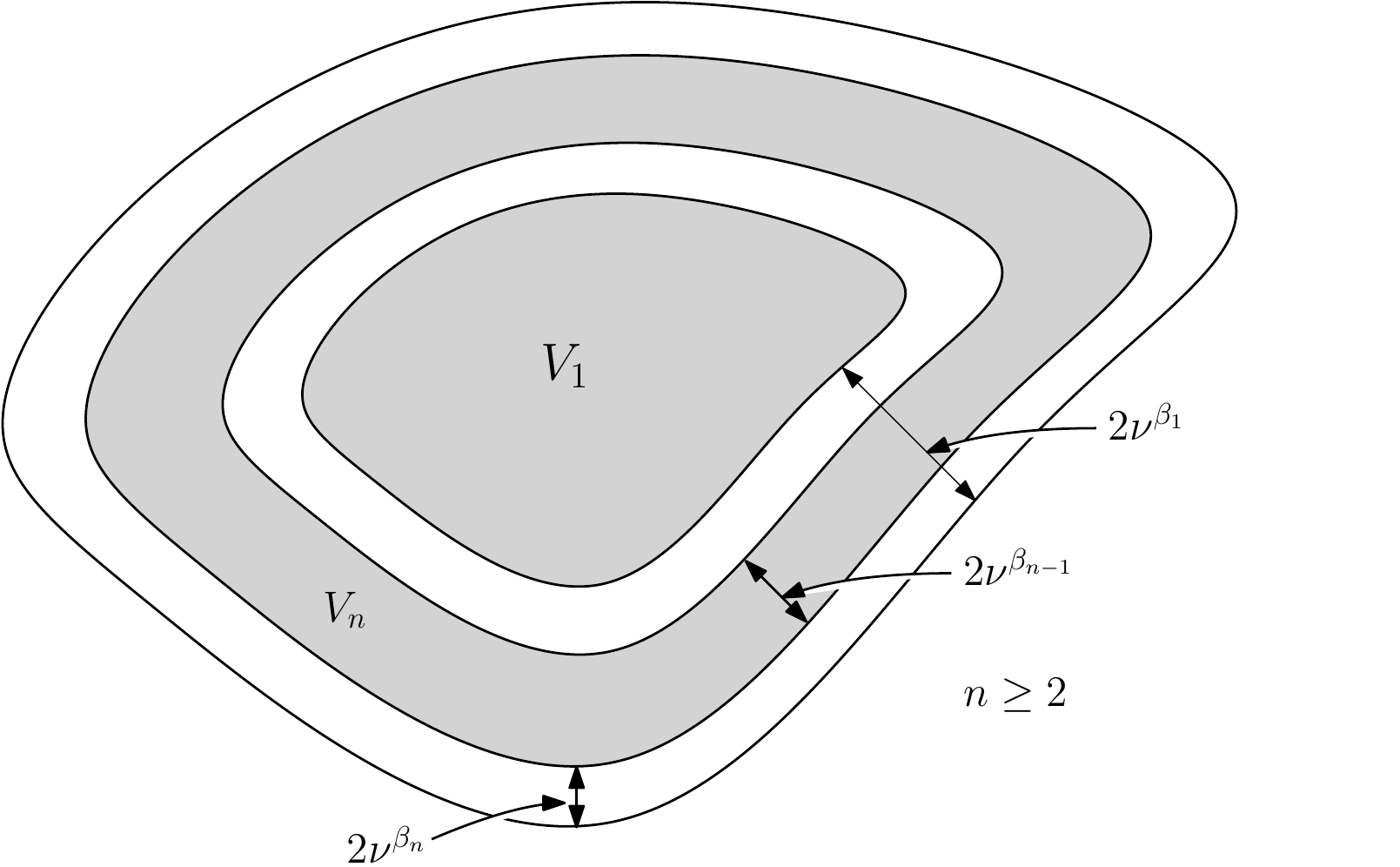}
  \caption{Decomposition of $\Omega$}
  \label{fig:decomp}
\end{figure}

It is easy to check that, for $\nu$ small enough, 
\be\la{rrp2}\ba {\rm meas}\, (V_{n}) \le C\nu^{\b_{n-1}},\quad{\rm meas}\,(V_{k}\cap V_{m})\le \left\{\ba &C\nu^{\b_{\max\{k,m\}}} \,\,\,{\rm if}\,\,\,|k-m|=1,\\
&0\,\,\,{\rm if}\,\,\,|k-m|>1.\ea\right. \ea
\ee
See Figure \ref{fig:foliation}.
\begin{figure}[h!]
  \centering
  \vspace{-4ex}
  \includegraphics[page=2,scale=0.87]{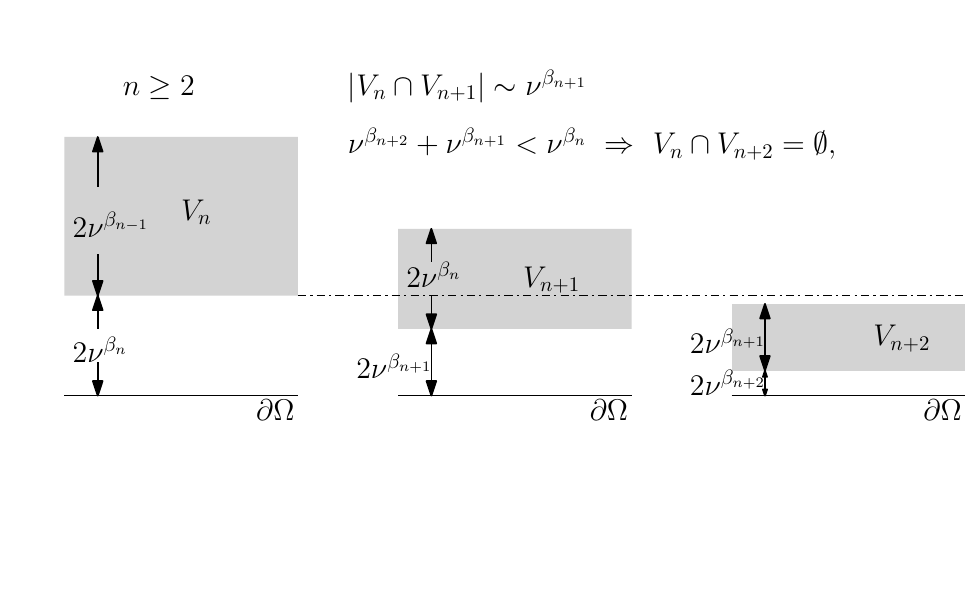}
  \vspace{-14ex}
  \caption{Foliation of $\p\Omega$}
  \label{fig:foliation}
\end{figure}

With the above decomposition, we have the following:
 \begin{proposition}\la{pp1} Let  $\{\xi_{n}\}_{n=1}^{N}$  be a  $C^{1}$ partition of unity   subordinate to  $\{V_{n}\}_{n=1}^{N}$  such that \be\la{u2}\ba  {\rm spt}\, \xi_{n}\subset V_{n},\quad 0\le \xi_{n}\le 1,\quad \sum_{n=1}^{N}\xi_{n}=1.\ea\ee
 Then, it follows that for $0\le n \le N$,
\be\la{ll1} \ba\na \left(\xi_{n}+\xi_{n+1}\right)^2 =0\,\,\,\text{ if}\,\,\,\,x\in V_{n}\cap V_{n+1},\quad{\rm and}\quad \na \xi_{n}=0\,\,\,\text{ if}\,\,\,\,x\in V_{n} \backslash \cup_{i\neq n} V_i,\ea\ee 
where we define
\be\la{y8} V_{N+1} := \left(\cup_{n=1}^{N}V_{n}\right)^{c}\quad{\rm and}\quad V_{0} :=\varnothing.\ee
\end{proposition}
\begin{proof}
From \eqref{u2} and \eqref{rrp2} we know that 
\[
\left(\xi_{n}+\xi_{n+1}\right)\big|_{V_{n}\cap V_{n+1}} \equiv 1, \qquad \xi_n \big|_{V_{n} \backslash \cup_{i\neq n} V_i} \equiv 1.
\]
Therefore \eqref{ll1} follows trivially by differentiating the above.
\end{proof}

\section{Regularization and resolved energy balance}\label{sec reg}
For $n = 1, \ldots, N$, define
\be\la{k1} \overline{f_{n}}(x,t) :=\left\{\ba & \int \eta_{\nu^{\b_{n}}}(x-y)f(y,t)dy, &x\in V_{n} \cap V_{n+1}^{c},\\
& \int \eta_{\nu^{\b_{n+1}}}(x-y)f(y,t)dy, &x\in V_{n} \cap V_{n+1}.\ea\right.\ee
From \eqref{rrp2} we know that
\be\la{f eq}
\overline{f_n} = \overline{f_{n+1}} \quad \text{on} \quad V_n \cap V_{n+1}.
\ee
From \eqref{2.12}-\eqref{b5}, it holds that, 
 for $p\in [1,\infty)$, 
 \be\la{j}   \|\na \overline{f_{n}}\|_{L^{p}(V_{n} )}\le \left\{\ba & C  \nu^{\b_{n}(\alpha-1)} \|f\|_{B_{p}^{\alpha,\infty}(\om^{\nu})},&x\in V_{n}\cap V_{n+1}^{c},\\
  & C  \nu^{\b_{n+1}(\alpha-1)} \|f\|_{B_{p}^{\alpha,\infty}(\om^{\nu})},&x\in V_{n}\cap V_{n+1}, \ea\right.\ee
 and 
\be\la{k2}   \| \overline{f_{n}}-f\|_{L^{p}(V_{n} )}\le \left\{\ba & C  \nu^{\alpha\b_{n}} \|f\|_{B_{p}^{\alpha,\infty}(\om^{\nu})},&x\in V_{n}\cap V_{n+1}^{c},\\
& C  \nu^{\alpha\b_{n+1}} \|f\|_{B_{p}^{\alpha,\infty}(\om^{\nu})},&x\in V_{n}\cap V_{n+1}.\ea\right.\ee

With \eqref{k1},  we deduce from \eqref{1ns}  that  
\bnn\ba & \p_{t}\overline{u^{\nu}_{n}} + \div \overline{(u^{\nu}\otimes u^{\nu})_{n}} + \na \overline{P^{\nu}_{n}} =\nu \lap \overline{u^{\nu}_{n}},\quad \forall\,\,x\in V_{n}\quad(n=1, \ldots, N).\ea
\enn
Multiplying it  by $\xi_{n}$ and summing up implies that  
\be\la{3.1}\p_{t}\left(\sum_{n=1}^{N}\xi_{n} \overline{u^{\nu}_{n}}\right)+ \sum_{n=1}^{N}\xi_{n} \div \overline{ (u^{\nu} \otimes u^{\nu})_n}+ \sum_{n=1}^{N}\xi_{n}\na \overline{P^{\nu}_{n}} =\nu\sum_{n=1}^{N}\xi_{n}\lap \overline{u^{\nu}_{n}},\quad \forall \,\,x\in \om\backslash \Ga_{2\nu}.
\ee
To deal with the boundary contribution,   we introduce a smooth  cut-off function $\th(x)$ in $\om$ such that \be\la{3.2} 0\le \th(x)\le 1,\quad \th(x)=1\,\,\ {\rm if}\,\, \,\, x\in \om^{4\nu},\quad   \th(x)=0\,\,\ {\rm if}\,\, x\notin  \om^{2\nu}, \quad \text{and}\ \ |\na \th|\le  4\nu^{-1}.\ee

Next we wish to test \eqref{3.1} by   $\th(x)\left(\sum_{n=1}^{N}\xi_{n} \overline{u^{\nu}_{n}}\right)$ to derive the resolved energy balance. However this test function fails to be solenoidal, and hence cannot be used as a legitimate test field for Leray-Hopf solutions. Therefore to make our argument work we must appeal to the following theorem:

\begin{theorem}[Theorem 1, \cite{mauro}] \label{thm_mauro} 
Assume that $\om$ is  an  open, bounded domain with   $C^{2}$ boundary $\p \om$, and $u$ is a Leray--Hopf solution of \eqref{1ns}. Then there exists a pressure field $P \in L^r(0,T; W^{1,s}(\om))$ with
\be\la{range}
{3\over s} + {2\over r} = 4, \qquad {4\over 3} < s < {3\over2},
\ee
such that for all $\varphi \in C^\infty_0((0, T)\times \om)$,
\begin{equation*}
\int^T_0\int_\om \left( u\cdot \partial_t\varphi + u\otimes u : \nabla \varphi + P \div\varphi + \nu u \cdot \Delta \varphi \right)\,dxdt = 0.
\end{equation*}
\end{theorem}

This way we can multiply \eqref{3.1} by   $\th(x)\left(\sum_{n=1}^{N}\xi_{n} \overline{u^{\nu}_{n}}\right)$ and integrate over $\Omega \times [0, T]$.  This then leads to
\be\la{3.3}  \ba
&\frac{1}{2}\int_{\om}\th \left(\sum_{n=1}^{N}\xi_{n} \overline{u^{\nu}_{n}}\right)^{2}(x,T)\,dx-\frac{1}{2}\int_{\om}\th \left(\sum_{n=1}^{N}\xi_{n} \overline{u^{\nu}_{n}}\right)^{2}(x,0)\,dx\\
&=\int_{0}^{T}\int_{\om}\th \left(\sum_{n=1}^{N}\xi_{n} \overline{u^{\nu}_{n}}\right)\left(\sum_{n=1}^{N} \xi_{n}\nu \lap \overline{u^{\nu}_{n}} -\sum_{n=1}^{N} \xi_{n} \div \overline{ (u^{\nu} \otimes u^{\nu})_{n}} -\sum_{n=1}^{N} \xi_{n}\na \overline{P^{\nu}_{n}}  \right)\,dxdt.\ea\ee

The main result of this section is the following.
\begin{proposition}[Resolved energy balance] \la{p}
Under the same hypotheses as in  Theorem \ref{thm0}, it holds that 
\be\la{9}\lim_{\nu\rightarrow 0} \int_{0}^{T}\int_{\om}\th \left(\sum_{n=1}^{N}\xi_{n} \overline{u^{\nu}_{n}}\right)\left(\sum_{n=1}^{N} \xi_{n}\nu \lap \overline{u^{\nu}_{n}} -\sum_{n=1}^{N} \xi_{n} \div \overline{ (u^{\nu} \otimes u^{\nu})_{n}} -\sum_{n=1}^{N} \xi_{n}\na \overline{P^{\nu}_{n}}  \right)\,dxdt=0.\ee
  \end{proposition}
The proof of Proposition \ref{p} is a direct consequence of Lemma \ref{lemma3.1} -- \ref{lemma3.3}  below.

\begin{lemma}[Resolved dissipation]\la{lemma3.1} 
Under the same hypotheses as in  Theorem \ref{thm0}, we have
\be\la{4.1} \lim_{\nu\rightarrow 0} \int_{0}^{T}\int_{\om}\th \left(\sum_{n=1}^{N}\xi_{n} \overline{u^{\nu}_{n}}\right)\left(\sum_{n=1}^{N} \xi_{n}\nu \lap \overline{u^{\nu}_{n}}\right)\,dxdt =0.\ee\end{lemma}
\begin{proof}
 Owing to   \eqref{rrp2} and \eqref{u2}, we find
\be\la{p21}  \xi_{k}\xi_{m}=0\quad {\rm if}\quad |k-m|\ge 2.\ee 
Integration by parts then gives   
 \be\ba\la{4.2} 
 &\nu\int_{0}^{T}\int_{\om}\th \left(\sum_{n=1}^{N}\xi_{n} \overline{u^{\nu}_{n}}\right)\left(\sum_{n=1}^{N} \xi_{n} \lap \overline{u^{\nu}_{n}}\right)\,dxdt =\nu\sum_{k,m=1}^{N} \int_{0}^{T}\int_{\om} \th  \xi_{k} \xi_{m} \overline{u^{\nu}_{k}}\lap \overline {u_{m}^{\nu}} \,dxdt \\
&=- \nu\sum_{|k-m|\le 1}  \int_{0}^{T}\int_{\om}  \th \xi_{k} \xi_{m} \na\overline{u^{\nu}_{k}}\na  \overline {u_{m}^{\nu}} \,dxdt - \nu\sum_{|k-m|\le 1}  \int_{0}^{T}\int_{\om}\na \th \xi_{k} \xi_{m}  \overline{u^{\nu}_{k}} \na  \overline {u_{m}^{\nu}} \,dxdt\\
&\quad- \nu\sum_{|k-m|\le 1}  \int_{0}^{T}\int_{\om}\th\na( \xi_{k} \xi_{m}) \overline{u^{\nu}_{k}} \na  \overline {u_{m}^{\nu}} \,dxdt. \ea\ee

The  terms on the right side of   \eqref{4.2} are  treated as follows.
First, it follows from  \eqref{h1},  \eqref{rrp2}, \eqref{u2},    \eqref{j}   that, if $|k-m|=0$,
%
\begin{align}
&\left| \nu  \int_{0}^{T}\int_{\om} \th \xi_{k}^{2}   \na\overline{u^{\nu}_{k}} \na \overline {u_{k}^{\nu}} \,dxdt \right| =\left| \nu  \int_{0}^{T}\left(\int_{V_{k}\cap V_{k+1}}+\int_{V_{k}\cap V_{k+1}^{c}}\right) \th \xi_{k}^{2}   \na\overline{u^{\nu}_{k}} \na \overline {u_{k}^{\nu}} \,dxdt \right| \nonumber \\
 &\le C   \nu \int_{0}^{T}\left( \|\xi_{k}^{2}\|_{L^{3}(V_{k}\cap V_{k+1})}\|\na \overline{u^{\nu}_{k}}\|_{L^{3}(V_{k}\cap V_{k+1})}^{2} +\|\xi_{k}^{2}\|_{L^{3}(V_{k}\cap V_{k+1}^{c})}\|\na \overline{u^{\nu}_{k}}\|_{L^{3}(V_{k}\cap V_{k+1}^{c})}^{2} \right) \,dt \nonumber\\
 &\le C  \nu\left( \nu^{\frac{1}{3}\b_{k+1}} \int_{0}^{T}  \|\na \overline{u^{\nu}_{k}}\|_{L^{3}(V_{k}\cap V_{k+1})}^{2} \,dt +\nu^{\frac{1}{3}\b_{k-1}} \int_{0}^{T}  \|\na \overline{u^{\nu}_{k}}\|_{L^{3}(V_{k}\cap V_{k+1}^{c})}^{2} \,dt \right) \nonumber \\
  &\le C  \nu\left( \nu^{\frac{1}{3}\b_{k+1}+2\b_{k+1}(\alpha-1)} +\nu^{\frac{1}{3}\b_{k-1}+2\b_{k}(\alpha-1)}\right) \int_{0}^{T} \|u^{\nu}\|_{B_{3}^{\alpha,\infty}(\om^{\nu})}^{2} \,dt \la{4.5}  \\
 &\le C \nu^{1+\frac{1}{3}\b_{k-1}+2\b_{k}(\alpha-1)}. \nonumber
\end{align} 
If  $|k-m|=1,$
  \be\la{4.5s}\ba 
  & \left| - \nu \int_{0}^{T}\int_{\om} \th \xi_{k} \xi_{m} \na\overline{u^{\nu}_{k}} \na \overline {u_{m}^{\nu}} \,dxdt \right|\\
 & \quad \le C  \nu \int_{0}^{T}\|\na \overline{u^{\nu}_{k}}\|_{L^{3}(V_{k}\cap V_{m})} \|\na \overline{u^{\nu}_{m}}\|_{L^{3}(V_{k}\cap V_{m})} \|\xi_{k}\xi_{m}\|_{L^{3}(V_{k}\cap V_{m})} \,dt \\
  & \quad \le C   \nu^{1+\b_{\max\{k,m\}}\left(\frac{1}{3}+2(\alpha-1)\right)} \int_{0}^{T} \|u^{\nu}\|_{B_{3}^{\alpha,\infty}(\om^{\nu})}^{2} \,dt \\
 & \quad \le C  \nu^{1+\b_{\max\{k,m\}}\left(2\alpha-\frac{5}{3}\right)}.\ea\ee
Thanks to  \eqref{h1}  and \eqref{r1a},  we know that
 \bnn 1+\frac{1}{3}\b_{k-1}+2\b_{k}(\alpha-1)>0, \quad \text{and} \quad 1+\b_{\max\{k,m\}}\left(2\alpha-\frac{5}{3}\right)>1-\b_{\max\{k,m\}}\ge0. \enn   
 Hence,   from \eqref{4.5}-\eqref{4.5s} we conclude
 \be\la{4.5d}\ba &\lim_{\nu\rightarrow0}\left|\sum_{|k-m|\le 1} \nu\int_{0}^{T}\int_{\om}\th \xi_{k} \xi_{m} \na\overline{u^{\nu}_{k}} \na \overline {u_{m}^{\nu}} \,dxdt \right|\\
 &\le \lim_{\nu\rightarrow0}\left(\sum_{|k-m|=0}+\sum_{|k-m|= 1}\right) \left|\nu\int_{0}^{T}\int_{\om}\th \xi_{k} \xi_{m} \na\overline{u^{\nu}_{k}} \na \overline {u_{m}^{\nu}} \,dxdt \right|\\
 &\le C \lim_{\nu\rightarrow0} \sum_{|k-m|\le 1}\left(\nu^{1+\frac{1}{3}\b_{k-1}+2\b_{k}(\alpha-1)}+  \nu^{1+\b_{\max\{k,m\}}\left(2\alpha-\frac{5}{3}\right)}\right) =0.\ea
\ee

Second, observe from \eqref{3.2} that $\th =0$ if $x\notin \om^{2\nu}\cap \Ga_{4\nu}$.
Then,  utilizing   \eqref{k2},  \eqref{h1}, \eqref{3.2},  and the  Hardy-type inequality,  it follows that,  for small $\nu$,
\begin{align}
& \sum_{|k-m|\le 1}\nu\int_{0}^{T}\int_{\om}\na \th \xi_{k}\xi_{m}\overline{u^{\nu}_{k}}\na \overline{u^{\nu}_{m}} \,dxdt =\nu\int_{0}^{T}\int_{\om^{2\nu}\cap \Ga_{4\nu}}\na \th \xi_{N}^{2}\overline{u^{\nu}_{N}}\na \overline{u^{\nu}_{N}} \,dxdt \nonumber\\
&\le \left( \nu \int_{0}^{T}\int_{\om^{2\nu}\cap \Ga_{4\nu}}|\na \overline{u^{\nu}_{N}}|^{2} \,dxdt  \right)^{\frac{1}{2}}
\left(\nu\int_{0}^{T}  \|\na \th\|_{L^{6}(\om^{2\nu}\cap \Ga_{4\nu})}^{2}\|\overline{u^{\nu}_{N}}-u^{\nu}\|_{L^{3}(V_{N})}^{2}+\|u^{\nu}\na \th\|_{L^{2}(\Ga_{4\nu})}^{2} \right)^{\frac{1}{2}} \nonumber\\
&\le  \left( \nu \int_{0}^{T}\int_{ \om}|\na \overline{u^{\nu}_{N}}|^{2} \,dxdt  \right)^{\frac{1}{2}}
\left(\nu\int_{0}^{T}\nu^{2\alpha-\frac{5}{3}}\|u^{\nu}\|_{B_{3}^{\alpha,\infty}(\om^{\nu})}^{2}+\|\na u^{\nu}\|_{L^{2}(\Ga_{4\nu})}^{2}\right)^{\frac{1}{2}}\nonumber\\
&\le C\left( \nu^{1+(2\alpha-\frac{5}{3})}+ \nu \int_{0}^{T}\int_{ \Ga_{4\nu}}|\na \overline{u^{\nu}_{N}}|^{2} \,dt \right)^{\frac{1}{2}}. \la{4.3}
\end{align}
Thanks to 
   \eqref{h1} and   \eqref{0j1},   we  take  $\nu\rightarrow0$ in  \eqref{4.3} to get  
\be\la{4.4}\ba &\lim_{\nu\rightarrow0}\left|\sum_{|k-m|\le 1}\nu\int_{0}^{T}\int_{\om}\na \th \xi_{k}\xi_{m}\overline{u^{\nu}_{k}}\na \overline{u^{\nu}_{m}} \,dxdt \right| =0.
\ea\ee

Finally,  thanks to  Proposition \ref{pp1}  and \eqref{f eq}, we infer that
\begin{align}
& \sum_{|k-m|\le 1} \nu \int_{0}^{T}\int_{\om} \th \na (\xi_{k}\xi_{m})\overline{u^{\nu}_{k}}\na \overline{u^{\nu}_{m}} \,dxdt \nonumber\\
&=\sum_{k=1}^{N} \nu \int_{0}^{T}\left(\int_{V_{k}\cap V_{k-1}}+\int_{V_{k}\cap V_{k+1}}\right) \th \na (\xi_{k}^{2})\overline{u^{\nu}_{k}}\na \overline{u^{\nu}_{k}} \,dxdt \nonumber\\
&\quad+\sum_{k=1}^{N-1} \nu \int_{0}^{T}\int_{V_{k}\cap V_{k+1}} \th \na (\xi_{k}\xi_{k+1})\left(\overline{u^{\nu}_{k+1}}\na \overline{u^{\nu}_{k}}+\overline{u^{\nu}_{k}}\na \overline{u^{\nu}_{k+1}} \right) \,dxdt  \la{4.10e}\\
&=\sum_{k=1}^{N} \nu \int_{0}^{T}\left(\int_{V_{k}\cap V_{k-1}}\th \na (\xi_{k}^{2})\overline{u^{\nu}_{k}}\na \overline{u^{\nu}_{k}} \,dx + \int_{V_{k}\cap V_{k+1}}\th \na (\xi_{k}^{2})\overline{u^{\nu}_{k+1}}\na \overline{u^{\nu}_{k+1}} \,dx \right)\,dt  \nonumber\\
&\quad+\sum_{k=1}^{N-1} \nu \int_{0}^{T}\int_{V_{k}\cap V_{k+1}} \th \na (2\xi_{k}\xi_{k+1})\overline{u^{\nu}_{k+1}}\na \overline{u^{\nu}_{k+1}} \,dxdt \nonumber\\
&=  \sum_{k=1}^{N-1} \nu \int_{0}^{T}\int_{V_{k}\cap V_{k+1}} \th \na \left(\xi_{k}^{2}+\xi_{k+1}^{2}+2\xi_{k}\xi_{k+1}\right)\overline{u^{\nu}_{k+1}}\na \overline{u^{\nu}_{k+1}} \,dxdt =0, \nonumber
\end{align}
where the second equality is due to \eqref{f eq},  and the third equality comes from relabeling and \eqref{ll1}. 

As a result of  \eqref{4.2}, \eqref{4.5d},   \eqref{4.4} and \eqref{4.10e}, we conclude  \eqref{4.1}.
\end{proof}

\begin{lemma}[Bulk energy flux]\la{lemma3.2} 
Under the same hypotheses as in  Theorem \ref{thm0}, we have
\be\la{3.5}  \lim_{\nu\rightarrow 0}\int_{0}^{T}\int_{\om}\th \left(\sum_{n=1}^{N}\xi_{n} \overline{u^{\nu}_{n}}\right)\left( \sum_{n=1}^{N}\xi_{n} \div \overline{ (u^{\nu} \otimes u^{\nu})_{n}}\right) \,dxdt =0.\ee\end{lemma}
\begin{proof} 
 By \eqref{p21},  integration by parts leads to 
 \be\la{3.6} 
 \begin{split}
 & \int_{0}^{T}\int_{\om}\th \left(\sum_{n=1}^{N}\xi_{n} \overline{u^{\nu}_{n}}\right)\left( \sum_{n=1}^{N }\xi_{n} \div \overline{ (u^{\nu} \otimes u^{\nu})_{n}}\right) \,dxdt \\
 & =\sum_{|k-m|\le 1}\int_{0}^{T}\int_{\om} \th \xi_{k}\xi_{m}\overline{u^{\nu}_{k}}  \div \overline{ (u^{\nu} \otimes u^{\nu})_{m}} \,dxdt \\
 &=\sum_{|k-m|\le 1} \int_{0}^{T}\int_{\om} \left(\overline{u^{\nu}_{m}}\otimes \overline{u^{\nu}_{m}}- \overline{ (u^{\nu} \otimes u^{\nu})_{m}}\right):\na (\th \xi_{k}\xi_{m}\overline{u^{\nu}_{k}}) \,dxdt \\
 & \quad - \sum_{|k-m|\le 1}\int_{0}^{T}\int_{\om} \overline{u^{\nu}_{m}}\otimes \overline{u^{\nu}_{m}}: \na (\th \xi_{k}\xi_{m}\overline{u^{\nu}_{k}}) \,dxdt.
 \end{split}
 \ee

We  claim
\be\la{9g}
\lim_{\nu\rightarrow 0} \sum_{|k-m|\le 1} \int_{0}^{T}\int_{\om} \left(\overline{u^{\nu}_{m}}\otimes \overline{u^{\nu}_{m}}- \overline{ (u^{\nu} \otimes u^{\nu})_{m}}\right):\na (\th \xi_{k}\xi_{m}\overline{u^{\nu}_{k}}) \,dxdt =0. \ee
In fact,  we notice from   \eqref{h1}, \eqref{02h2},  \eqref{3.2},  \eqref{0},    \eqref{k2}, and  the  Hardy-type inequality that
%
\begin{align}
&\left| \sum_{|k-m|\le 1}\int_{0}^{T}\int_{\om}\left(\overline{u^{\nu}_{m}}\otimes \overline{u^{\nu}_{m}}- \overline{ (u^{\nu} \otimes u^{\nu})_{m}}\right):(\na \th) \xi_{k}\xi_{m}\overline{u^{\nu}_{k}} \,dxdt \right| \nonumber\\
 &\le  C\left(\int_{0}^{T} \int_{\Ga_{2\nu} \cap \om^{\nu}} \left|\overline{u^{\nu}_{N}} \right|^{4} \,dxdt\right)^{\frac{1}{2}}\left(\int_{0}^{T} \int_{\Ga_{4\nu} \cap \om^{2\nu}}|\na \th (\overline{u^{\nu}_{N}}-u^{\nu}+u^{\nu})|^{2} \, dxdt \right)^{\frac{1}{2}} \nonumber \\
 &\le  C \left(\nu\int_{0}^{T} \|u^{\nu}\|_{L^{\infty}(\Ga_{4\nu})}^{4} \,dt \right)^{\frac{1}{2}}\left(\int_{0}^{T} \nu^{-\frac{5}{3}} \| \overline{u^{\nu}_{N}}-u^{\nu}\|_{L^{3}}^{2} \,dt + \int_{0}^{T} \int_{\Ga_{4\nu} }|\na \th u^{\nu}|^{2} \,dxdt \right)^{\frac{1}{2}} \nonumber \\
  &\le  C \nu^{\frac{1}{2}}\left(\int_{0}^{T} \nu^{(2\alpha-\frac{5}{3})} \| u^{\nu} \|_{B_{3}^{\alpha,\infty}(\om^{\nu})}^{2} \,dt +\int_{0}^{T} \int_{\Ga_{4\nu} }|\na u^{\nu}|^{2} \,dxdt \right)^{\frac{1}{2}} \la{03.8j}\\
 &\le  C\left( \nu^{1+(2\alpha-\frac{5}{3})} +\nu\int_{0}^{T} \int_{\Ga_{4\nu} }|\na u^{\nu}|^{2} \,dxdt \right)^{\frac{1}{2}}. \nonumber
\end{align}
This, together with  \eqref{h1} and \eqref{0j1}, implies that    
\be\la{ui3}\ba    
\lim_{\nu\rightarrow0}  \sum_{|k-m|\le 1}\int_{0}^{T}\int_{\om}\left(\overline{u^{\nu}_{m}}\otimes \overline{u^{\nu}_{m}}- \overline{ (u^{\nu} \otimes u^{\nu})_{m}}\right): (\na \th) \xi_{k}\xi_{m}\overline{u^{\nu}_{k}} \,dxdt =0.\ea\ee
Next,  from   \eqref{u2},  \eqref{h1}, \eqref{j}, and  Lemma \ref{lem2.1}  we have
\bnn\ba   
& \sum_{k}\int_{0}^{T}\int_{\om}\left(\overline{u^{\nu}_{k}}\otimes \overline{u^{\nu}_{k}}- \overline{ (u^{\nu} \otimes u^{\nu})_{k}}\right):  \th  \xi_{k}^{2}\na \overline{u^{\nu}_{k}} \,dxdt \\
&\le C \sum_k \int_{0}^{T} \|\overline{u^{\nu}_{k}}\otimes \overline{u^{\nu}_{k}}- \overline{ (u^{\nu} \otimes u^{\nu})_{k}}\|_{L^{\frac{3}{2}}(V_{k}\cap V_{k+1})}  \|\na \overline{u^{\nu}_{k}}\|_{L^{3}(V_{k}\cap V_{k+1})} \,dt \\
&\quad+ C \sum_k \int_{0}^{T}\|\overline{u^{\nu}_{k}}\otimes \overline{u^{\nu}_{k}}- \overline{ (u^{\nu} \otimes u^{\nu})_{k}}\|_{L^{\frac{3}{2}}(V_{k}\cap V_{k+1}^{c})}  \|\na \overline{u^{\nu}_{k}}\|_{L^{3}(V_{k}\cap V_{k+1}^{c})} \,dt \\
&\le C \sum_k \int_{0}^{T}\left(\nu^{2\beta_{k+1} \alpha +\beta_{k+1}(\alpha-1)}+\nu^{2\beta_{k} \alpha +\beta_{k}(\alpha-1)} \right)\| u^{\nu} \|_{B_{3}^{\alpha,\infty}(\om^{\nu})}^{3} \,dt \\
&\le  C \sum_k \nu^{\b_{k}(3\alpha-1)} .\ea\enn
Similarly, 
\bnn \left| \sum_{|k-m|=1}\int_{0}^{T}\int_{\om}\left(\overline{u^{\nu}_{m}}\otimes \overline{u^{\nu}_{m}}- \overline{ (u^{\nu} \otimes u^{\nu})_{m}}\right):  \th  \xi_{k}\xi_{m}\na \overline{u^{\nu}_{k}} \,dxdt \right|\le  C \sum_{|k-m|=1} \nu^{\b_{\max\{m,k\}}(3\alpha-1)}.\enn
 The above two inequalities  and \eqref{h1}   guarantee that   
\be\la{b3sde}  
\lim_{\nu\rightarrow0} \left|\sum_{|k-m|\le 1}\int_{0}^{T}\int_{\om}\left(\overline{u^{\nu}_{m}}\otimes \overline{u^{\nu}_{m}}- \overline{ (u^{\nu} \otimes u^{\nu})_{m}}\right):  \th  \xi_{k}\xi_{m}\na \overline{u^{\nu}_{k}} \,dxdt \right|=0.\ee
By Proposition \ref{pp1},  the same deduction  as  \eqref{4.10e} yields that
\begin{align}
&\sum_{|k-m|\le 1}\int_{0}^{T}\int_{\om}\left(\overline{u^{\nu}_{m}}\otimes \overline{u^{\nu}_{m}}- \overline{ (u^{\nu} \otimes u^{\nu})_{m}}\right):  \th \na (\xi_{k}\xi_{m})\overline{u^{\nu}_{k}} \,dxdt \nonumber \\
&= \sum_{k=1}^{N-1}\int_{0}^{T}\int_{V_{k}\cap V_{k+1}}\left(\overline{u^{\nu}_{k+1}}\otimes \overline{u^{\nu}_{k+1}}- \overline{ (u^{\nu} \otimes u^{\nu})_{k+1}}\right):    \th \na \left(\xi_{k}^{2}+\xi_{k+1}^{2}+2\xi_{k}\xi_{k+1}\right) \,dxdt \nonumber \\
&=0. \la{ui2} 
\end{align}
   As a result of   \eqref{ui3}--\eqref{ui2},  we conclude \eqref{9g}.

It remains to  control  the last integral   appeared  in \eqref{3.6}.  By the fact  $\div \overline{u^{\nu}_{n}}=0$, we deduce that,  if   $|k-m|=0$,
\begin{align}
& \sum_{|k-m|=0}\int_{0}^{T}\int_{\om} \overline{u^{\nu}_{m}}\otimes \overline{u^{\nu}_{m}}: \na (\th \xi_{k}\xi_{m}\overline{u^{\nu}_{k}}) \,dxdt = \frac{1}{2}\sum_{k=1}^N\int_{0}^{T}\int_{\om}| \overline{u^{\nu}_{k}}|^{2} \overline{u^{\nu}_{k}}\cdot \na (\th \xi^{2}_{k}) \,dxdt \nonumber \\ 
& = {1\over2} \sum_{k=1}^N \int_0^T \int_{V_k} |\overline{u^{\nu}_{k}}|^{2} \xi_k^2 \overline{u^{\nu}_{k}}\cdot \na \th \,dxdt + {1\over2} \sum_{k=1}^N \int_0^T \int_{V_k}  \th |\overline{u^{\nu}_{k}}|^{2} \overline{u^{\nu}_{k}}\cdot \na (\xi^{2}_{k}) \,dxdt \nonumber  \\
& = {1\over2} \sum_{k=1}^N \int_0^T \int_{V_k} |\overline{u^{\nu}_{k}}|^{2} \xi_k^2 \overline{u^{\nu}_{k}}\cdot \na \th \,dxdt + {1\over2} \sum_{k=1}^N \int_0^T \int_{V_{k-1} \cap V_k}  \th |\overline{u^{\nu}_{k}}|^{2} \overline{u^{\nu}_{k}}\cdot \na (\xi^{2}_{k}) \,dxdt \nonumber \\
& \quad \ + {1\over2} \sum_{k=1}^N \int_0^T \int_{V_{k} \cap V_{k+1}}  \th |\overline{u^{\nu}_{k+1}}|^{2} \overline{u^{\nu}_{k+1}}\cdot \na (\xi^{2}_{k}) \,dxdt \label{03.8jkq} \\
& = {1\over2} \sum_{k=1}^N \int_0^T \int_{V_k} |\overline{u^{\nu}_{k}}|^{2} \xi_k^2 \overline{u^{\nu}_{k}}\cdot \na \th \,dxdt + {1\over2} \sum_{k=1}^{N-1} \int_0^T \int_{V_{k} \cap V_{k+1}}  \th |\overline{u^{\nu}_{k}}|^{2} \overline{u^{\nu}_{k}}\cdot \na (\xi^{2}_{k} + \xi^2_{k+1}) \,dxdt \nonumber \\
& \quad \ + {1\over2} \int_0^T \left( \int_{V_0 \cap V_1} \th |\overline{u^{\nu}_{1}}|^{2} \overline{u^{\nu}_{1}}\cdot \na (\xi^{2}_{1}) \,dx  + \int_{V_N \cap V_{N+1}} \th |\overline{u^{\nu}_{N}}|^{2} \overline{u^{\nu}_{N}}\cdot \na (\xi^{2}_{N}) \,dx  \right) \,dt \nonumber \\
& = {1\over2} \sum_{k=1}^N \int_0^T \int_{V_k} |\overline{u^{\nu}_{k}}|^{2} \xi_k^2 \overline{u^{\nu}_{k}}\cdot \na \th \,dxdt - \sum_{k=1}^{N-1} \int_0^T \int_{V_{k} \cap V_{k+1}}  \th |\overline{u^{\nu}_{k}}|^{2} \overline{u^{\nu}_{k}}\cdot \na (\xi_{k} \xi_{k+1}) \,dxdt,  \nonumber 
\end{align}
where the third equality is due to \eqref{f eq} and \eqref{ll1};
and if $|k-m|=1$,
\begin{align}
& \sum_{|k-m|=1}\int_{0}^{T}\int_{\om} \overline{u^{\nu}_{m}}\otimes \overline{u^{\nu}_{m}}: \na( \th \xi_{k}\xi_{m}\overline{u^{\nu}_{k}}) \,dxdt \nonumber \\
= & \sum_{k=1}^{N-1} \int^T_0 \int_{V_k \cap V_{k+1}} \overline{u^{\nu}_{k+1}}\otimes \overline{u^{\nu}_{k+1}}: \na( 2 \th \xi_{k}\xi_{k+1}\overline{u^{\nu}_{k+1}}) \,dxdt \nonumber \\
= & \sum_{k=1}^{N-1} \int^T_0 \int_{V_k \cap V_{k+1}} |\overline{u^{\nu}_{k+1}}|^2  \overline{u^{\nu}_{k+1}} \cdot \na (\th \xi_{k}\xi_{k+1}) \,dxdt \la{9a} \\
= & \sum_{k=1}^{N-1} \int^T_0 \int_{V_k \cap V_{k+1}} \Big( |\overline{u^{\nu}_{k+1}}|^2  \overline{u^{\nu}_{k+1}} \cdot \na \th (\xi_{k}\xi_{k+1}) + \th |\overline{u^{\nu}_{k+1}}|^2  \overline{u^{\nu}_{k+1}} \cdot \na (\xi_{k}\xi_{k+1}) \Big) \,dxdt \nonumber \\
= & \sum_{k=1}^{N-1} \int^T_0 \int_{V_k \cap V_{k+1}} \Big( |\overline{u^{\nu}_{k}}|^2  \overline{u^{\nu}_{k}} \cdot \na \th (\xi_{k}\xi_{k+1}) + \th |\overline{u^{\nu}_{k}}|^2  \overline{u^{\nu}_{k}} \cdot \na (\xi_{k}\xi_{k+1}) \Big) \,dxdt. \nonumber
\end{align}

Putting the above calculations  together and applying Lemma \ref{lemma2.3} we have 
\begin{align}
& \sum_{|k-m|\le 1}\int_{0}^{T}\int_{\om} \overline{u^{\nu}_{m}}\otimes \overline{u^{\nu}_{m}}: \na (\th \xi_{k}\xi_{m}\overline{u^{\nu}_{k}}) \,dxdt \nonumber\\
= &\ {1\over2} \sum_{k=1}^N \int_0^T \int_{V_k} |\overline{u^{\nu}_{k}}|^{2} \xi_k^2 \overline{u^{\nu}_{k}}\cdot \na \th \,dxdt + \sum_{k=1}^{N-1} \int^T_0 \int_{V_k \cap V_{k+1}} |\overline{u^{\nu}_{k}}|^2  \overline{u^{\nu}_{k}} \cdot \na \th (\xi_{k}\xi_{k+1}) \,dxdt \nonumber \\
\le &\ C \int^T_0 \int_{\Ga_{2\nu} \cap \om^{\nu}} |\overline{u^{\nu}_{k}}|^{2} | \overline{u^{\nu}_{k}} \cdot \na \th | \,dxdt \label{3.2 3} \\
\le &\ C \left( \nu \int^T_0 \| u^{\nu} \|_{L^\infty({\Ga_{4\nu}})}\,dt \right)^{1/2} \left( \int^T_0 \int_{\Ga_{4\nu}} | \na u^\nu|^2 \,dxdt \right)^{1/2}.        \nonumber
\end{align}
From \eqref{02h2} and \eqref{0j1} we conclude that
\begin{equation}\la{3.17r}
\lim_{\nu \to 0} \sum_{|k-m|\le 1}\int_{0}^{T}\int_{\om} \overline{u^{\nu}_{m}}\otimes \overline{u^{\nu}_{m}}: \na (\th \xi_{k}\xi_{m}\overline{u^{\nu}_{k}}) \,dxdt = 0.
\end{equation}

Taking \eqref{3.6}--\eqref{9g}, and  \eqref{3.17r} into account,     we  complete the proof  of Lemma \ref{lemma3.2}. 
\end{proof}

\begin{lemma}\la{lemma3.3} 
Under the same hypotheses as in  Theorem \ref{thm0}, we have
\be\la{5.1} \lim_{\nu\rightarrow 0} \int_{0}^{T}\int_{\om}\th \left(\sum_{n=1}^{N}\xi_{n} \overline{u^{\nu}_{n}}\right)\left( \sum_{n=1}^{N} \xi_{n}\na \overline{P^{\nu}_{n}} \right) dxdt=0.\ee
\end{lemma}
\begin{proof}
The proof   is a  slight modification of  that in   Lemma \ref{lemma3.2}, and hence we omit it here.
\end{proof}

\section{Proof of Theorem \ref{thm0}}\label{sec proof}
  \subsection{Vanishing of global dissipation}
  We aim to prove the validity of  \eqref{00}.  The combination of   \eqref{0} with \eqref{3.3} generates
 \be\la{85}\ba &0\le 2\nu\int_{0}^{T}\int_{\om}|\na u^{\nu}|^{2}\\
 &\le   \left[ \int_{\om} |u_{0}^{\nu}|^{2} - \int_{\om}  \th \left(\sum_{n=1}^{N}\xi_{n} \overline{(u^{\nu}_{0})_{n}}\right)^{2} \right] +  \left[ \int_{\om}\th \left|\sum_{n=1}^{N}\xi_{n} \overline{u^{\nu}_{n}}\right|^{2}-\int_{\om} |u^{\nu}|^{2} \right]\\
&\quad+2\int_{0}^{T}\int_{\om}\th \left(\sum_{n=1}^{N}\xi_{n} \overline{u^{\nu}_{n}}\right)\left(\sum_{n=1}^{N} \xi_{n}\nu \lap \overline{u^{\nu}_{n}} -\sum_{n=1}^{N} \xi_{n} \div \overline{ (u^{\nu} \otimes u^{\nu})_{n}} -\sum_{n=1}^{N} \xi_{n}\na \overline{P^{\nu}_{n}}  \right)\\
&=: I+I\!I+I\!I\!I.
\ea\ee

First,  from Proposition \ref{p} it follows that
\be\la{l1}\lim_{\nu\rightarrow0}I\!I\!I=0.\ee

Next,  basic  properties of mollifier $\eta_{\nu^{\b}}$ ensure that \bnn\ba  \int_{\om}  \th\xi_{1}^{2} |\overline{u^{\nu}_{1}}|^{2}dx&\le \int_{V_{1}}    |\overline{u^{\nu}_{1}}|^{2} \le \int_{\om}|u^{\nu} |^{2} .\ea\enn 
Then,
\be\la{po}\ba I\!I
&=\sum_{2<k+m}\int_{\om}  \th \xi_{k}\xi_{m} \overline{u^{\nu}_{k}}\,\overline{u^{\nu}_{m}}+\int_{\om} \th\xi_{1}^{2} \overline{u^{\nu}_{1}}^{2}- \int_{\om} |u^{\nu}|^{2}\\
&\le \sum_{2<k+m}\int_{\om}  \th \xi_{k}\xi_{m} \overline{u^{\nu}_{k}}\,\overline{u^{\nu}_{m}}\\
&\le C \left(\int_{V_{k}\cap V_{m}}  |\overline{u^{\nu}_{k}}|^{2}\right)^{\frac{1}{2}}\left(\int_{V_{k}\cap V_{m}}  |\overline{u^{\nu}_{m}}|^{2}\right)^{\frac{1}{2}}\le C \int_{\Ga_{2\nu^{\b_{1}}}}  |u^{\nu}|^{2} dx,\ea\ee
and hence the uniform bound of   $u^{\nu}$ in $L^{\infty}(0,T;L^{2})$ implies that 
\be\la{l2} \ba \lim_{\nu\rightarrow0}I\!I
\le   C \lim_{\nu\rightarrow0}\int_{\Ga_{2\nu^{\b_{1}}}}  |u^{\nu}|^{2} dx=0.
\ea\ee

Finally,  since
\bnn\ba  
&\left|\int_{\om} |u_{0}|^{2}- \int_{\om} \th \xi_{1}^{2} |\overline{(u^{\nu}_{0})_{1}}|^{2}\right| =\left|\int_{\om} |u_{0}|^{2}- \int_{\om^{\nu^{\b_1}}}   \xi_{1}^{2} |\overline{(u^{\nu}_{0})_{1}}|^{2}\right|\\
&\le \int_{\om \backslash \om^{\nu^{\b_1}}} |u_{0}|^{2} +\int_{\om^{\nu^{\b_1}}}   (1-\xi_{1}^{2} )|\overline{(u^{\nu}_{0})_{1}}|^{2}+\left|\int_{\om^{\nu^{\b_1}}}\left( |u_{0}|^{2}-|\overline{(u^{\nu}_{0})_{1}}|^{2}\right)\right| \\
&\le 2\|u_{0}\|_{L^{2}(\Ga_{2\nu^{\b_1}})}  +C\|u_{0}\|_{L^{2}(\om)} \|u_{0}-\overline{(u^{\nu}_{0})_{1}}\|_{L^{2}(\om^{\nu^{\b_1}})}\\
&\le 2\|u_{0}\|_{L^{2}(\Ga_{2\nu^{\b_1}})}  +C \left( \|u_{0}-\overline{(u_{0})_{1}}\|_{L^{2}(\om^{\nu^{\b_1}})}+ \|u_{0}- u^{\nu}_{0}\|_{L^{2}(\om^{\nu^{\b_1}})}\right)\\
&\to 0 \quad\text{ as } \nu\to 0, 
\ea\enn
then  the strong convergence of $u_{0}^{\nu}$ to $u_{0}$ in $L^{2}$ guarantees that
\bnn  \lim_{\nu\rightarrow0}\left|\int_{\om} |u_{0}^{\nu}|^{2}- \int_{\om} \th \xi_{1}^{2} |\overline{(u^{\nu}_{0})_{1}}|^{2}\right|  = 0.
\enn  This allows us to deduce that
 \be\la{l3}\ba \lim_{\nu\rightarrow0}|I| &= \lim_{\nu\rightarrow0} \left|\int_{\om} |u_{0}^\nu|^{2}- \sum_{k,m=1}^{N} \int_{\om}  \th \xi_{k}\xi_{m} \overline{(u^{\nu}_{0})_{k}}\,\overline{(u^{\nu}_{0})_{m}}\right| \\
&\le \lim_{\nu\rightarrow0} \left|\int_{\om} |u_{0}^\nu|^{2}- \int_{\om} \th \xi_{1}^{2} |\overline{(u^{\nu}_{0})_{1}}|^{2}\right|+\sum_{2<k+m} \lim_{\nu\rightarrow0} \left|\int_{\om}  \th \xi_{k}\xi_{m} \overline{(u^{\nu}_{0})_{k}}\,\overline{(u^{\nu}_{0})_{m}}\right|=0,
\ea\ee
where in  the last  equality we have used
\bnn\ba \sum_{2<k+m} \lim_{\nu\rightarrow0} \left|\int_{\om}  \th \xi_{k}\xi_{m} \overline{(u^{\nu}_{0})_{k}}\,\overline{(u^{\nu}_{0})_{m}}\right| =0,\ea\enn
which comes from \eqref{po}--\eqref{l2}.
As a result of \eqref{85}--\eqref{l1} and  \eqref{l2}--\eqref{l3}, we obtain the desired  \eqref{00}.

 \subsection{Convergence to Euler solutions.}

  Under the    assumptions in Theorem \ref{thm0} and Lemma \ref{lemma2.2},    there is some $(u,P)$ such that, upon to some subsequence, \be\la{80} u^{\nu}\rightharpoonup u\,\,\, {\rm in}\,\,\,L^{3}(0,T;B_{3}^{\si,\infty}(\om^{\nu}))\cap L^{\infty}(0,T;L^{2}(\om)),\quad P^{\nu}\rightharpoonup P\,\,\, {\rm in}\,\,\,L^{\frac{3}{2}}(0,T;L^{\frac{3}{2}}(\om)).\ee
Thanks to   \eqref{80}  and  \eqref{0}, we have 
$$\p_{t} u^{\nu}=\nu\lap u^{\nu}-\na P^{\nu}-\div (u^{\nu}\otimes u^{\nu})\in L^{\frac{3}{2}}(0,T;W^{-1,\frac{3}{2}}(\om)),$$
 and moreover,  
\be\la{81} 
u^{\nu}\rightarrow u\,\,\, {\rm in}\,\,\,L^{3}(0,T;L^{3}(\om^{\nu})) \cap C\left([0,T],L^{2}_{weak}(\om)\right)
\ee 
owing to  the compactness results. 
In addition,  it follows from   \eqref{0}  that, as $\nu\rightarrow0,$
\be\la{80a} 
\left|\int_{0}^{T}\int_{\om}\nu \na u^{\nu}\cdot \na \varphi \right|\le   C\nu^{\frac{1}{2}}\left(\nu\int_{0}^{T}\|\na u^{\nu}\|_{L^{2}(\om)}^{2}\right)^{\frac{1}{2}}\rightarrow 0.
\ee
Having \eqref{80}--\eqref{80a} in hand,   we easily check that  $u$ solves Euler equations \eqref{1e}   in $\om\times (0,T).$

\section {Boundary layers for smoother solutions}\label{sec higher reg}

The boundary layer $\Ga_{4\nu}$ in \eqref{0j1} of Theorem \ref{thm0} in fact holds for all $\alpha>\frac{1}{3}$. The emphasis of the previous analysis lies in determining the boundary layer for solutions near the critical Onsager's regularity. On the other hand, when the solutions are more regular, the hypotheses in Theorem \ref{thm0}  can be relaxed, and the boundary layer can be even thinner, as is shown in the following theorem.

\begin{theorem}\la{thm2}  
Let  $\om\subset \r$ be a bounded domain with $C^2$ boundary.  Let $\{ u^\nu \}_{\nu>0}$ be a sequence of Leray--Hopf weak solutions to \eqref{1ns} with initial data $u_{0}^{\nu}$ and suppose that $u_{0}^{\nu} \to u_0$ in $L^{2}(\om)$ as $\nu \to 0$. Assume in addition that \eqref{h1} holds, and 
\be\ba\la{02h2s} 
\left\{\ba&   u^{\nu} \,\,{\rm  is\,\, uniformly \,\,in}\,\,  \nu \,\,{\rm bounded\,\, in}\, \,  L^{4}\left(0,T;L^{p}\left(\Ga_{4\nu}\right)\right),\\
&  P^{\nu}\,\,{\rm  is \,\, uniformly \,\,in}\,\,  \nu \,\,{\rm bounded\,\, in}\, \,  L^{2}\left(0,T;L^{\frac{p}{2}}\left(\Ga_{4\nu}\right)\right),\ea\right.\ea
\ee
with $p>\frac{6}{3\alpha-1}$. Let $a>1$ be such that 
\be\la{es}
a<\frac{3}{5-6\alpha}, \ \text{when }\ {1\over3} < \alpha < {5\over6}; \qquad a<\infty, \ \text{when }\ {5\over6} \le \alpha < {1}.
\ee
If
\be\la{0j1s}  
\lim_{\nu\rightarrow0}\nu \int_{0}^{T}\int_{\Ga_{4\nu^{a}}}|\na u^{\nu}|^{2}dxdt=0,
\ee 
then,  the global   viscous dissipation vanishes, i.e., \eqref{00} holds true.  Moreover,  $u^{\nu}$ converges locally  in $L^{3}(0,T;L^{3}(\om))$,  up  to a  subsequence,  to a weak solution   of  Euler equations  \eqref{1e}. 
\end{theorem}

\begin{remark}
As $\alpha \rightarrow \frac{1}{3}^+,$  Theorem \ref{thm2} recovers Theorem \ref{thm0}. But as $\alpha$ increases, we can relax the regularity requirement on the boundary, cf. \eqref{02h2s}, and the thickness of the boundary layer becomes $\nu^a$ with $a>1$. In particular, as $\alpha \to 1^-$, the thickness becomes arbitrarily small.
\end{remark}

\begin{proof}
The idea of the proof is very similar to the one explained before, and hence we will only focus on the ingredients  different from those  in the proof of Theorem \ref{thm0}. 

First of all, we modify the construction of the increasing and finite sequence $\{\b_{n}\}_{n=1}^{N}$ in Section \ref{subsec BL} as
\be\la{y10} 
0=\b_{0}<\b_{1}<\cdot\cdot\cdot<\b_{N-1}\le 1<\b_{N}<\left\{\ba&\frac{3}{5-6\si},\,\,\,\,{\rm if}\,\,\,\alpha\in \left(\frac{1}{3},\frac{5}{6}\right),\\
&\infty,\,\,\,\,{\rm if}\,\,\,\alpha\in \left[\frac{5}{6},1\right),\ea\right.
\ee
and 
\bnn  
\b_{n}< \frac{1}{2(1-\alpha)}\left(1+\frac{1}{3}\b_{n-1}\right).
\enn
In fact, here we consider the case of $\b_{N} = a >1$ which  satisfies  \eqref{y10}, in stead of $\b_{N} = 1$  defined in \eqref{r1a}.

The ``pealed-off'' set $V_{N+1}$ in \eqref{y8} now becomes 
\be\la{y8s}  
\Ga_{\nu^{a}} :=V_{N+1}= \left(\cup_{n=1}^{N}V_{n}\right)^c.
\ee
In addition, the near boundary layer cut-off function $\th$ in \eqref{3.2} is modified as 
\be\la{3.2s} 0\le \th(x)\le 1,\quad \th(x)=1\,\,{\rm if}\,\, \,\, x\in \om^{2\nu^{a}},\quad   \th(x)=0\,\,{\rm if}\,\, x\notin  \om^{\nu^{a}}, \quad |\na \th|\le  2\nu^{-a}.\ee 

With the above preparations,  to complete the proof of Theorem \ref{thm2},  we only need to check the following:

\medskip

(a) \textbf {Inequality \eqref{4.3} in Lemma \ref{lemma3.1}.} \\

In Theorem \ref{thm2},  it can be treated as   
\be\la{4.3s}\ba 
& \sum_{|k-m|\le 1}\nu\int_{0}^{T}\int_{\om}\na \th \xi_{k}\xi_{m}\overline{u^{\nu}_{k}}\na \overline{u^{\nu}_{m}}\\
& =\nu\int_{0}^{T}\int_{\om^{2\nu^{a}}\cap \Ga_{4\nu^{a}}}\na \th \xi_{N}^{2}\overline{u^{\nu}_{N}}\na \overline{u^{\nu}_{N}}\\
&\le  \left( \nu \int_{0}^{T}\int_{ \om}|\na \overline{u^{\nu}_{N}}|^{2}  \right)^{\frac{1}{2}}
\left(\nu\int_{0}^{T}\nu^{a(2\alpha-\frac{5}{3})} \|u^{\nu}\|_{B_{3}^{\alpha,\infty}(\om^{\nu^{a}})}^{2}+\|\na u^{\nu}\|_{L^{2}(\Ga_{2\nu^{a}})}^{2}\right)^{\frac{1}{2}}\\
&\le C\left( \nu^{1+a(2\alpha-\frac{5}{3})}+ \nu \int_{0}^{T}\int_{ \Ga_{4\nu^{a}}}|\na \overline{u^{\nu}_{N}}|^{2} \right)^{\frac{1}{2}}\rightarrow 0,
\ea\ee 
 provided   
\bnn 
1+a\left(2\alpha-\frac{5}{3}\right)>0,
\enn 
which is valid owing to \eqref{y10}.

\medskip 
 
 (b) \textbf {Inequality \eqref{03.8j} in Lemma \ref{lemma3.2}.}\\
 
In Theorem \ref{thm2},  we estimate \eqref{03.8j}  as below
 \be\la{03.8js}\ba    
 &\left| \sum_{|k-m|\le 1}\int_{0}^{T}\int_{\om}\left(\overline{u^{\nu}_{m}}\otimes \overline{u^{\nu}_{m}}- \overline{ (u^{\nu} \otimes u^{\nu})_{m}}\right):\na \th \xi_{k}\xi_{m}\overline{u^{\nu}_{k}}\right|\\
 &\le  C\left(\int_{0}^{T} \int_{\Ga_{4\nu^{a}} \cap \om^{2\nu^{a}}} \left|\overline{u^{\nu}_{N}} \right|^{4}\right)^{\frac{1}{2}}\left(\int_{0}^{T} \nu^{-\frac{5}{3}a} \| \overline{u^{\nu}_{N}}-u^{\nu}\|_{L^{3}}^{2}+\int_{0}^{T} \int_{\Ga_{4\nu^{a}} }|\na \th u^{\nu}|^{2}\right)^{\frac{1}{2}}\\
 &\le  C \left(\nu^{a(1-\frac{4}{p})}\int_{0}^{T} \|u^{\nu}\|_{L^{p}(\Ga_{8\nu^{a}})}^{4}\right)^{\frac{1}{2}} \left(\int_{0}^{T} \nu^{a(2\alpha-\frac{5}{3})} \| u^{\nu} \|_{B_{3}^{\si,\infty}(\om^{\nu^{a}})}^{2}+\int_{0}^{T} \int_{\Ga_{4\nu^{a}} }|\na u^{\nu}|^{2}\right)^{\frac{1}{2}}\\
&\le  C(T)\nu^{\frac{1}{2}a(1-\frac{4}{p}+2\alpha-\frac{5}{3})}  +C\nu^{\frac{1}{2}(a(1-\frac{4}{p})-1)}\left(\nu\int_{0}^{T} \int_{\Ga_{4\nu^{a}} }|\na u^{\nu}|^{2}\right)^{\frac{1}{2}}\rightarrow0,
\ea\ee
provided 
\bnn 
a \left(1-\frac{4}{p}+2\alpha-\frac{5}{3} \right)>0\quad {\rm and}\quad a \left(1-\frac{4}{p} \right)-1\ge0,
\enn
which holds true due to \eqref{es} and \eqref{y10}.

\medskip
 
(c) \textbf {Inequality \eqref{3.2 3} in Lemma \ref{lemma3.2}.}\\

This can be achieved from a similar argument as that in deriving \eqref{03.8js}. 
\end{proof}

 \bigskip
 
\section*{Acknowledgement} 
R. M. Chen would like to thank Theodore Drivas for helpful discussions. The work of R. M. Chen is partially supported by National Science Foundation under Grant DMS-1907584. The work of Z. Liang is partially supported by the fundamental research funds for central universities (JBK 1805001). The work of D. Wang is partially supported by the National Science Foundation under grants DMS-1613213 and DMS-1907519.
 \bigskip

\bibliographystyle{siam}
\bibliography{reference}

\end{document}